\documentclass[11pt,reqno]{article}

\usepackage{amsmath,amsthm,amsfonts,amssymb,latexsym,mathrsfs,color}
\usepackage{hyperref}

\newtheorem{theorem}{Theorem}
\newtheorem{corollary}[theorem]{Corollary}

\newcommand{\lrf}[1]{\lfloor #1\rfloor}

\title{On some binomial coefficients related to the evaluation of $\tan(nx)$}

\author
{Shi-Mei Ma \footnote{ {\it Email address:}
shimeima@yahoo.com.cn (S.-M. Ma)} }
\date{\footnotesize School of Mathematics and Statistics,
        Northeastern University at Qinhuangdao,\\ Hebei 066004,
        China}

\begin{document}

\maketitle

\begin{abstract}
The purpose of this paper is to study some binomial coefficients which are related to the evaluation of $\tan(nx)$. We present a connection between these binomial coefficients and the coefficients of
a family of derivative polynomials for tangent and secant.
\end{abstract}
\section{Introduction}
Denote by $D$ the differential operator ${d}/{d x}$. Throughout this paper, set $y=\tan(x)$ and $z=\sec(x)$.
Then $D(y)=z^2$ and $D(z)=yz$. An important tangent identity is given by
$$1+y^2=z^2.$$
In 1995, Hoffman~\cite{Hoffman95} considered two sequences of {\it derivative polynomials} defined respectively by
\begin{equation*}\label{derivapoly-1}
D^n(y)=P_n(y)\quad {\text and}\quad D^n(z)=z Q_n(y)
\end{equation*}
for $n\geq 0$.
From the chain rule it follows that the polynomials $P_n(y)$ satisfy $P_0(y)=y$ and $P_{n+1}(y)=(1+y^2)P_n'(y)$, and similarly $Q_0(y)=1$ and $Q_{n+1}(y)=(1+y^2)Q_n'(y)+yQ_n(y)$.
Various refinements of the derivative polynomials have been pursued by several authors
(see~\cite{Franssens07,Ma12,Ma122} for instance).

In 1972,
Beeler {\it et al.} found the following elegant identity~\cite[Item~16]{Beeler}:
\begin{equation}\label{Beeler}
\tan (n\arctan (t))=\frac{1}{i}\frac{(1+it)^n-(1-it)^n}{(1+it)^n+(1-it)^n}  \quad {\text for}\quad n\geq 0,
\end{equation}
where $i=\sqrt{-1}$.
Let
\begin{equation*}\label{RnkTnk}
R(n,k)=\binom{n}{2k+1}\quad {\text and}\quad T(n,k)=\binom{n}{2k}.
\end{equation*}
We can now present the following equivalent version of~\eqref{Beeler}:
\begin{equation*}\label{Beeler}
\tan (n x)=\frac{\sum_{k=0}^{\lrf{\frac{n-1}{2}}}(-1)^kR(n,k)\tan^{2k+1}(x)}{\sum_{k=0}^{\lrf{\frac{n}{2}}}(-1)^kT(n,k)\tan^{2k}(x)},
\end{equation*}
where $x=\arctan (t)$ (see~\cite[A034839,~A034867]{Sloane} for details).
In the sequel to the work of Beeler {\it et al.}, many other methods have been given to compute $\tan (n x)$.
Several of them have in common the use of angle addition formula. For example, Szmulowicz~\cite{Szmulowicz05} obtained a generalized tangent angle addition formula.
The purpose of this paper is to explore some further applications of the numbers $R(n,k)$ and $T(n,k)$.

Using the following recurrence relations~\cite[p.~10]{Comtet74}:
\begin{equation*}\label{binom}
\binom{n}{k}=\binom{n-1}{k-1}+\binom{n-1}{k}\quad {\text and}\quad \binom{n}{k}=\frac{n-k+1}{k}\binom{n}{k-1},
\end{equation*}
it can be easily verified that
\begin{equation}\label{Rnk-recu}
nR(n+1,k)=(n+2k+1)R(n,k)+(n-2k+1)R(n,k-1)
\end{equation}
and
\begin{equation}\label{Tnk-recu}
nT(n+1,k)=(n+2k)T(n,k)+(n-2k+2)T(n,k-1).
\end{equation}

For $n\geq 0$, we always assume that
$$(Dz)^{n+1}(z)=(Dz)(Dz)^n(z)=D(z(Dz)^n(z))$$
and
$$(Dz)^{n+1}(y)=(Dz)(Dz)^n(y)=D(z(Dz)^n(y)).$$
In this paper we consider the expansions of $(Dz)^n(z)$ and $(Dz)^n(y)$, where the numbers $R(n,k)$ and $T(n,k)$ appear in a natural way.
\section{Polynomials related to $(Dz)^n(z)$ and $(Dz)^n(y)$}
For $n\geq 0$, we define
\begin{equation}\label{def-derivative-1}
(Dz)^n(z)=\sum_{k=0}^{\lrf{\frac{n}{2}}}M(n,k)y^{n-2k}z^{n+2k+1}
\end{equation}
and
\begin{equation}\label{def-derivative-2}
(Dz)^n(y)=\sum_{k=0}^{\lrf{\frac{n+1}{2}}}N(n,k)y^{n-2k+1}z^{n+2k}.
\end{equation}

For example, when $n=1$, we have $$(Dz)(z)=D(z^2)=2yz^2\quad {\text and}\quad (Dz)(y)=D(zy)=y^2z+z^3.$$
\begin{theorem}\label{thm-1}
For $0\leq k\leq \left\lfloor{\frac{n}{2}}\right\rfloor$, the numbers $M(n,k)$ satisfy the recurrence relation
\begin{equation}\label{recurrence-1}
M(n+1,k)=(n+2k+2)M(n,k)+(n-2k+2)M(n,k-1)
\end{equation}
with the initial conditions $M(0,0)=1$ and $M(0,k)=0$ for $k\geq 1$,
and the numbers $N(n,k)$ satisfy the recurrence relation
\begin{equation}\label{recurrence-11}
N(n+1,k)=(n+2k+1)N(n,k)+(n-2k+3)N(n,k-1)
\end{equation}
with the initial conditions $N(0,0)=1$ and $N(0,k)=0$ for $k\geq 1$.
\end{theorem}
\begin{proof}
Note that
\begin{align*}
(Dz)^{n+1}(z)&=(Dz)(Dz)^n(z)\\
             &=\sum_{k\geq 0}(n+2k+2)M(n,k)y^{n-2k+1}z^{n+2k+2}+
\sum_{k\geq 0}(n-2k)M(n,k)y^{n-2k-1}z^{n+2k+4}.
\end{align*}
Thus we obtain~\eqref{recurrence-1}.
Similarly, we get~\eqref{recurrence-11}.
\end{proof}

Combining~\eqref{Rnk-recu}, \eqref{Tnk-recu}, \eqref{recurrence-1} and~\eqref{recurrence-11}, we get the following result.
\begin{corollary}
For $0\leq k\leq \left\lfloor{\frac{n}{2}}\right\rfloor$, we have
$M(n,k)=n!R(n+1,k)$
and $N(n,k)=n!T(n+1,k)$.
\end{corollary}

Recall that $z^2=1+y^2$, we define
\begin{equation}\label{Rny-def}
(Dz)^n(z)=\begin{cases}
(2m)!zR_{2m+1}(y)& \text{if $n=2m$,}\\
 (2m+1)!R_{2m+2}(y)& \text{if $n=2m+1$;}
\end{cases}
\end{equation}
and
\begin{equation*}
(Dz)^n(y)=\begin{cases}
(2m)!T_{2m+1}(y)& \text{if $n=2m$,}\\
 (2m+1)!zT_{2m+2}(y)& \text{if $n=2m+1$,}
\end{cases}
\end{equation*}
where $m\geq 0$.
We now present explicit formulas for the polynomials $R_n(y)$ and $T_n(y)$.
\begin{theorem}
For $n\geq 1$, we have
\begin{equation}\label{Rny}
R_n(y)=\sum_{k=0}^{\lrf{\frac{n-1}{2}}}R(n,k)y^{n-2k-1}(1+y^2)^{\lrf{\frac{n}{2}}+k}
\end{equation}
and
\begin{equation}\label{Tny}
T_n(y)=\sum_{k=0}^{\lrf{\frac{n}{2}}}T(n,k)y^{n-2k}(1+y^2)^{\lrf{\frac{n-1}{2}}+k}.
\end{equation}
\end{theorem}
\begin{proof}
We only prove the explicit formula for $R_{2m+1}(y)$ and the others can be proved in a similar way.
Combining~\eqref{def-derivative-1} and~\eqref{Rny-def}, we obtain
$$(Dz)^{2m}(z)=(2m)!\sum_{k\geq 0}R(2m+1,k)y^{2m-2k}z^{2m+2k+1}.$$
Therefore, we get
\begin{align*}
R_{2m+1}(y)&=\sum_{k\geq 0}R(2m+1,k)y^{2m-2k}z^{2m+2k}\\
           &=\sum_{k\geq 0}R(2m+1,k)y^{2m-2k}(1+y^2)^{m+k},
\end{align*}
and the statement immediately follows.
\end{proof}

Using~\eqref{Rny} and~\eqref{Tny}, the first few terms of $R_n(y)$ and $T_n(y)$ can be calculated directly as follows:
$$R_1(y)=1,R_2(y)=2y+2y^3,R_3(y)=1+5y^2+4y^4,R_4(y)=4y+16y^3+20y^5+8y^7;$$
$$T_1(y)=y,T_2(y)=1+2y^2,T_3(y)=3y+7y^3+4y^5,T_4(y)=1+9y^2+16y^4+8y^6.$$

For $n\geq 1$, we define
\begin{equation*}
\widetilde{R}_n(y)=\begin{cases}
T_{2m}(y)& \text{if $n=2m$,}\\
R_{2m+1}(y)& \text{if $n=2m+1$;}
\end{cases}
\end{equation*}
and
\begin{equation*}
\widetilde{T}_n(y)=\begin{cases}
R_{2m}(y)& \text{if $n=2m$,}\\
T_{2m+1}(y)& \text{if $n=2m+1$.}
\end{cases}
\end{equation*}

Let $\widetilde{R}_n(y)=\sum_{k=1}^{n}\widetilde{R}(n,k)y^{2k-2}$ and
$\widetilde{T}_n(y)=\sum_{k=1}^{n}\widetilde{T}(n,k)y^{2k-1}$.
For $1\leq n\leq 5$, the coefficients of $\widetilde{R}_n(y)$
can be arranged as follows with $\widetilde{R}(n,k)$ in row $n$ and column $k$:
$$\begin{array}{ccccccc}
  1 &  &  &  & & &\\
  1 & 2 &  &  & & &\\
  1 & 5 & 4 &  & & &\\
  1 & 9 & 16 & 8 & &  &\\
  1 & 14 & 41 & 44 & 16 & &\\
\end{array}$$
For $1\leq n\leq 5$, the coefficients of $\widetilde{T}_n(y)$
can be arranged as follows with $\widetilde{T}(n,k)$ in row $n$ and column $k$:
$$\begin{array}{ccccccc}
  1 &  &  &  & & &\\
  2 & 2 &  &  & & &\\
  3 & 7 & 4 &  & & &\\
  4 & 16 & 20 & 8 & &  &\\
  5 & 30 & 61 & 52 & 16 & &\\
\end{array}$$
It should be noted that the number $\widetilde{R}(n,k)$ is the number of {\it $k$-part order-consecutive partition} of the set $\{1,2,...,n\}$ (see~\cite[\textsf{A056242}]{Sloane}).
The numbers $\widetilde{T}(n,k)$ appear as~\textsf{A210753} in~\cite{Sloane}.


\end{document}